\documentclass[12pt]{article}
\usepackage{amsmath,amsthm,amssymb,latexsym,color}
\usepackage{bbm}

\usepackage{fullpage}

\usepackage[colorlinks = true,linkcolor = blue,citecolor = blue,urlcolor = blue]{hyperref}
\usepackage{enumerate}

\let\oldproofname=\proofname
\renewcommand{\proofname}{\rm\bf{\oldproofname}}

\title{On partition functions of Gaussian random variables}
\author{{O. Friedland} \and {O. Gu\'edon} \and F. Souli
}

\newcommand\address{\noindent\leavevmode
	
	\medskip
	
	\noindent
	Omer Friedland
	\\
	Sorbonne Université,
	\\
	IMJ-PRG UMR7586, F-75005 Paris, France
	\\
	\texttt{\small%
		e-mail:  ofriedlan@imj-prg.fr}
	
	\medskip
	
	\noindent
	Olivier Gu\'{e}don \\
	Univ Gustave Eiffel, Univ Paris Est Creteil, CNRS \\
	LAMA UMR8050 F-77447 Marne-la-Vallée, France\\
	\texttt{\small%
		e-mail:  olivier.guedon@univ-eiffel.fr}
	
	\medskip

	\noindent
	Fabien Souli \\
	Univ Gustave Eiffel, Univ Paris Est Creteil, CNRS \\
	LAMA UMR8050 F-77447 Marne-la-Vallée, France\\
	\texttt{\small%
		e-mail: {souli.fabien@gmail.com}
	}
	
}

\newcommand{\E}{\mathbb E}
\newcommand{\R}{\mathbb R}
\newcommand{\Prob}{\mathbb P}
\newcommand{\one}{\mathbf 1}
\newcommand{\e}{\mathrm e}

\newcommand{\norm}[1]{\left\lVert #1\right\rVert}
\newcommand{\ip}[2]{\left\langle #1,#2\right\rangle}

\theoremstyle{plain}

\newtheorem{theorem}{Theorem}[section]
\newtheorem{lemma}[theorem]{Lemma}
\newtheorem{remark}[theorem]{Remark}
\newtheorem{proposition}[theorem]{Proposition}
\newtheorem{corollary}[theorem]{Corollary}

\begin{document}

\maketitle

\begin{abstract}
 For $\beta \in \R$, we define the partition function 
$$
Z_\beta(X) = \sum_{i=1}^N \exp(\beta X_i)
$$ 
of a centered Gaussian random vector  $X=(X_1, \ldots, X_N)$  with $\E[X_i^2]= 1$.
For $q \in \R$, we prove a complete phase
transition for the generalized $L_q$-means of $Z_\beta(X)$ at $q=1$. 
The centered Gaussian vector with covariance matrix $\Delta_N$ obtained from a regular simplex configuration of unit vectors 
maximizes these means for $q<1$, and  minimizes them for
$q>1$. 

The two regimes are governed by different principles. 
For $q > 1$ {and $\beta\ne0$}, the moment functional is globally strongly convex on the
entire set of correlation matrices, with an explicit modulus of convexity and a 
quantitative centroid-shape
stability estimate. For $q<1$, the strategy is different. We prove a universal comparison for log-concave
{profiles of reverse Brascamp--Lieb type.} Specializing {this result} to {the} Gumbel profile yields {a Laplace-transform} comparison between the partition functions. 
\end{abstract}

\bigskip

{\small
}

\section{Introduction}
\subsection{Presentation.}
Many problems in geometry involve describing the set that maximizes or minimizes certain Gaussian functionals under various 
constraints. These include{, for example,} the classical isoperimetry problem \cite{Borell75, SudakovTsirelson}.
In {convex geometry},  applications of the {Brascamp--Lieb} inequalities 
\cite{BrascampLieb} and reverse {Brascamp--Lieb} inequalities \cite{BartheReverse}  show that the regular simplex is the convex body that maximizes  the 
mean width among all convex bodies whose {maximal-volume inscribed ellipsoid} is the Euclidean unit ball \cite{Barthe1998} and {minimizes the mean width among all convex bodies whose minimal-volume enclosing ellipsoid is the Euclidean unit ball}
\cite{Schmuckenschlager1998}.

Other shape optimization problems are naturally connected to the study of Gaussian processes. For instance, maximizing the mean width of the convex hull of $N$ points 
constrained to lie on the unit sphere of the ambient space ${\R}^d$ is, in general, an open problem; see \cite[Section 9.10.2]{G-Klee94}. 
The case $N = 
d+1$ was called the {``Simplex Mean Width Conjecture''} and was solved recently by Mulgund \cite{Mulgund26}.  In this specific setting, his result asserts that the extreme configuration occurs when the $N$ points on the unit sphere form the 
vertices of a regular simplex. 
Detailed discussions of the  subject can be found in \cite{Litvak2018, Balakrishnan1963} and \cite[Chapter 10]{Weber1987}. It is also {explained in detail} in \cite{Mulgund26}, with further consequences in {signal theory}. 

In the study of comparison inequalities for Gaussian processes, generalizing Chevet's \cite{chevet1976}, Slepian's \cite{slepian1962}, Fernique's \cite{Fernique1974}, and Gordon's inequalities \cite{gordon1992}, it has proven natural to replace the max function by the
log-sum-exp function, which corresponds to the log of the partition function. 
There are other  possibilities {for replacing min-max functions or the sum of the $k$ largest coordinates by suitable approximations.}
{We refer, e.g.,} to \cite{chatterjee, peccati_turchi_2023, GS26} for further references.
Any centered Gaussian vector $X = (X_1, \ldots, X_N)$  such that $\E X_i^2 = 1$
has a covariance matrix $C$ that is the Gram matrix of a configuration of unit vectors,
also called a correlation matrix (and {conversely}). Therefore, 
to any correlation matrix $C$, we  associate the random variable $Z_\beta(X)$  given by
\begin{equation}
	\label{eq:defZbeta}
{Z}_\beta(X) := \sum_{i=1}^N e^{\beta X_i}
\end{equation}
where  $X \sim \mathcal{N}(0, C)$ and $\beta \in \R$.

The main result of the paper, Theorem \ref{thm:simplex-phase}, {exhibits a phase transition of} the generalized $L_q$-means of the partition function at $q = 1$.
It identifies the regular-simplex correlation matrix $\Delta_N$
as a minimizer
 when $q \in (1, + \infty)$ while it is a maximizer when $q \in (-\infty, 1)$.  
This phase transition is not a formal change of sign. 
The two regimes are governed by different principles. 

For $q>1$ {and $\beta\ne0$}, the relevant moment functional is globally strongly
convex on the full set of correlation matrices.
Its first variation at $\Delta_N$  measures
centroid defect, while its uniform second variation measures shape defect.
This yields both extremality and {a} quantitative stability result{;} see Theorem {\ref{thm:simplex-stability-above-one}}.
The argument is based on an evaluation of the second derivative of the functional along the classical interpolation path between two {centered Gaussian vectors}.

For $q<1$, we prove a universal comparison for
products of non-centered log-concave {profiles}. We simplify and generalize {Mulgund's approach} \cite{Mulgund26}.
The first step involves  {inf-convolution} and convexity  arguments combined with reverse {Brascamp--Lieb} inequalities due to Nakamura and Tsuji \cite{NT26}.
This leads to {the general statement in} Theorem \ref{thm:generalversionMNT}. 
The second part shows that the hypothesis $R - \frac{1}{N} J_N \succeq 0$ is used only at the very end of the argument, where it is incorporated through a simple linear-algebraic step, see Corollary \ref{cor:logconcave-product}.
Specializing the comparison to {the} Gumbel profile yields {a Laplace-transform} comparison between the partition functions, see \eqref{eq:LaplaceIneq}.

\subsection{Setup and main results}\label{sec:main-results}
Throughout, $N\ge2$, $\R^N$ is equipped with the standard scalar product 
$\langle \cdot,  \cdot \rangle$ and Euclidean norm defined by $\|y\|^2 = \langle y, y \rangle$. {The space of real symmetric $N\times N$ matrices} is equipped with the scalar product induced by the trace, the Frobenius norm $\| M \|_F^2 := \mathrm{tr}(M^T M) = \mathrm{tr}(M^2)$ and {the Loewner order}. {We write} $\one := (1,\ldots,1)^T$ {and} $J_N := \one\one^T$. Let
$$
\mathcal E_N := \{C\succeq0:C_{ii} = 1, 1\le i\le N\}
$$
be the set of correlation matrices.
The correlation matrix obtained from a regular simplex configuration of unit vectors is
\begin{align}
	\label{eq:simplex-covariance}
	\Delta_N := \frac{N}{N-1}I_N-\frac1{N-1}J_N,
\end{align}
where $I_N$ denotes the $N \times N$ identity matrix.
It has unit diagonal, off-diagonal entries $-1/(N-1)$, rank $N-1$, and
kernel $\operatorname{span}\{\one\}$. We denote by $S$ a centered Gaussian vector 
with covariance matrix $\Delta_N$. 

To any correlation matrix $C\in\mathcal E_N$, we associate the random variable
$Z_\beta(X)$ defined by \eqref{eq:defZbeta} where $X \sim \mathcal{N}(0, C)$.
For a positive random variable $Y$, set
\begin{align}\label{eq:q-mean-definition}
	\norm{Y}_q := 
	\begin{cases}
		(\E [Y^q])^{1/q},&q\ne0, \\
		\exp(\E [\log Y]),&q = 0.
	\end{cases}
\end{align}
We use this notation for generalized $L_q$-means; it is not a norm when
$q<1$. All quantities in \eqref{eq:q-mean-definition} are finite for the
Gaussian partition functions $Z_\beta(X)$ considered here.

Our first main theorem provides a phase transition of the functional
$C \mapsto \|Z_\beta(X)\|_q$, where $C \in \mathcal{E}_N$ and $X \sim \mathcal{N}(0, C)$, at $q = 1$. Since each correlation matrix is {the Gram matrix of a configuration of unit vectors}, we prove that for $q<1$, the $L_q${-mean} of the partition function is {maximized} when the configuration of the vertices 
is a regular simplex {whereas it is minimized} when $q >1$. {For $q=0$, $\log\|Z_\beta(X)\|_0=\E[\log Z_\beta(X)]$ is the expected log-partition function. For $\beta>0$, we use the normalized free-energy convention $\beta^{-1}\E[\log Z_\beta(X)]$, with temperature $1/\beta$. The quantity}
\[
{\frac{1}{\beta}} \E \left[ \log \left( \sum_{i = 1}^N e^{\beta X_i} \right) \right]
\]
{gives a finite-temperature Gaussian mean width for $\beta>0$, since}
\[
\lim_{\beta \to +\infty} \E \left[ \frac{1}{\beta}\log \left( \sum_{i = 1}^N e^{\beta X_i} \right) \right]
= 
\E [\max_{1 \le i \le N} X_i].
\]
\begin{theorem}[Simplex phase transition]\label{thm:simplex-phase}
Let $C\in\mathcal E_N$, $X \sim  \mathcal{N}(0,C)$ and $S \sim \mathcal{N}(0, \Delta_N)$. Let $\beta \in \R${. Then,} for every $q\in\R$,
\begin{align}\label{eq:q-mean-phase}
	\begin{cases}
		\norm{ Z_\beta(X)}_q
		\le \norm{ Z_\beta(S)}_q,&q<1, \\
		\norm{ Z_\beta(X)}_1 = N\e^{\beta^2/2},&q = 1, \\
		\norm{ Z_\beta(X)}_q
		\ge \norm{ Z_\beta(S)}_q,&q>1.
	\end{cases}
\end{align}
At $q = 0$, {taking logarithms of the generalized-mean inequality gives}
\begin{align}\label{eq:free-energy-main}
	\E\ [\log Z_\beta(X)]
	\le \E [\log Z_\beta(S)].
\end{align} 
\end{theorem}

In the case $q>1$, we use variational arguments to study the covariance functional $C\mapsto\E [Z_\beta(X)^q]$ where $X \sim \mathcal{N}(0, C)$. 
We prove a general statement about the second derivative of the functional along the {classical} 
interpolation path between two centered
Gaussian vectors{;} see Theorem \ref{thm:second-derivative}, from which we deduce the
next two statements.
\begin{theorem}[Global strong covariance convexity]
\label{thm:covariance-convexity}
Let $q>1$, $\beta \in \R^*$, and $C_0,C_1\in\mathcal E_N$. For any $t \in [0,1]$, put
$$
A = C_1-C_0,
\quad
C_t = (1-t)C_0+tC_1,
$$
and denote by $X_t$ a centered Gaussian random vector with covariance matrix $C_t$.
Then for $0\le t\le1$,
\begin{align}\label{eq:global-strong-convexity-chord}
	\E [Z_\beta(X_t)^q] \le
	(1-t)\E [Z_\beta(X_0)^q]
	+t\E [Z_\beta(X_1)^q] -\frac{\mu_{N,q,\beta}}2
	t(1-t)\norm{A}_{{F}}^2
\end{align}
where $\mu_{N,q,\beta} > 0$ is defined in \eqref{eq:mu-definition}.
In particular, $C\mapsto\E [Z_\beta(X)^q]$, where $X \sim \mathcal{N}(0, C)$, is strictly convex on the set of correlation matrices.
\end{theorem}
The next {theorem} is a precise quantitative stability result.
\begin{theorem}[Simplex stability above one]\label{thm:simplex-stability-above-one}
Let $q>1$, $\beta \in \R^*$, and let $C\in\mathcal E_N$. Let $X \sim \mathcal{N}(0, C)$ and 
$S \sim \mathcal{N}(0, \Delta_N)${. Then}
\begin{align}\label{eq:centroid-shape-stability}
	\E [Z_\beta(X)^q]-\E [Z_\beta(S)^q]
	\ge
	\eta_{N,q,\beta}\one^TC\one
	+\frac{\mu_{N,q,\beta}}2\norm{C-\Delta_N}_{{F}}^2{,}
\end{align}
where $\mu_{N,q,\beta} > 0$  and $\eta_{N,q,\beta} >0$ are defined in \eqref{eq:mu-definition}
and \eqref{eq:eta-definition}.
\end{theorem}
	If $C$ is the Gram matrix of unit vectors $v_1,\ldots, v_N$, then
	\begin{align}\label{eq:centroid-gram}
		\one^TC\one =  \langle C {\one}, {\one} \rangle = \norm{\sum_{i = 1}^N v_i}^2,
	\end{align}
	so the first term in \eqref{eq:centroid-shape-stability} measures
	non-centering while the second {term} measures shape defect.

For $q < 1$, the approach is {different}. We start by proving an extension of a
reverse {Brascamp--Lieb-type} inequality for non-centered log-concave {profiles}. 
\begin{theorem}
\label{thm:generalversionMNT}
Let $w_1, \ldots, w_N : \R \to (0, +\infty)$ be log-concave functions.
For $1\le i\le N$, define
\begin{align} \label{eq:phi}
	\varphi_i(u) := \log\int_{\R}w_i(u-x)d\gamma(x),
\end{align}
and
\begin{align} \label{eq:psi}
	\psi_i(s) :=
	(t^2/2) \square \varphi_i =  \inf_{t\in\R} \left\{ \frac{t^2}{2}+\varphi_i(s-t) \right\},
\end{align}
where
$$
d\gamma(x) = \frac1{\sqrt{2\pi}}\e^{-x^2/2}dx.
$$
Let $R$ be an $N\times N$ correlation matrix, let $X\sim {\mathcal N}(0,R)$, and put
$$
T = R^{1/2}.
$$
Then
\begin{align} \label{eq:variational}
	\log \E \left[ \prod_{i = 1}^Nw_i(X_i) \right] \ge \sup_{y\in\R^N} \left\{ \sum_{i = 1 }^N\psi_i((Ty)_i)-\frac12\norm{y}^2 \right\}. 
\end{align}
\end{theorem}
{It is worth noting that equality holds when each $w_i$ has a log-affine profile}
$w_i(z) = e^{\alpha_i z}$, $\alpha_i \in \R$.
The proof is based on reverse {Brascamp--Lieb} inequalities for centered log-concave functions due to Nakamura and Tsuji \cite{NT26}{;} see also \cite{MNT26}. 
{Whereas} Mulgund \cite{Mulgund26} {considered only} functions of the form 
$w = {\mathbbm{1}}_{(-\infty,c]}$, we 
{observe that inf-convolution and standard arguments from convex analysis}
lead to this {more} general statement.
{This formulation} simplifies and generalizes his approach.

An important corollary of Theorem \ref{thm:generalversionMNT} is the following family of inequalities,
 valid under the assumption that $R-\frac1N J_N \succeq 0${.} 
 {Its proof uses the simple linear-algebra argument of Lemma \ref{lem:linearalgebra}.}
\begin{corollary}
	\label{cor:logconcave-product}
Let $R$ be a correlation matrix such that
\[
R-\frac1N J_N \succeq 0{,}
\]
and {let} $X \sim \mathcal{N} (0, R)$ {be} a centered Gaussian {vector with} covariance matrix $R$.
Then for all log-concave functions $w_1, \ldots, w_N : \R \to (0, + \infty)$,
\begin{equation}
	\label{eq:stocorder2}
	\E \left[ \prod_{i=1}^{N} w_i(X_i) \right]
	\ge
	\exp  \left(  \sup_{t \in \R} \left\{ - \frac{N t^2}{2} + 
	\sum_{i=1}^{N} \psi_i (t) \right\} \right)
\end{equation}
where the $\psi_i$'s are defined in \eqref{eq:psi}.
{If $w_1=\cdots=w_N=w$, where $w : \R \to [0, + \infty)$}, we get
\begin{equation}
	\label{eq:stocordermain1}
	\E \left[ \prod_{i=1}^{N} w(X_i) \right]
	\ge
	\E \left[ \prod_{i=1}^{N} w(G_i) \right] =  ( \E \left[ w(G_1) \right])^N, 
\end{equation}
where $G=(G_1, \ldots, G_N) \sim \mathcal{N}(0, I_N)$ is a standard Gaussian random vector.  In particular, for any $\lambda > 0$
and $\beta \in \R$,
\begin{equation}
	\label{eq:LaplaceIneq}
	\E \left[ \exp\left( -\lambda	\sum_{i=1}^N e^{\beta X_i} \right) \right]
	\ge
	\E  \left[	\exp\left( -\lambda	\sum_{i=1}^N e^{\beta G_i} \right) \right]
\end{equation}
\end{corollary}
\begin{remark}
For an arbitrary real number $c$, set $w = \mathbbm{1}_{(-\infty, c]}${. Then} $w$ is log-concave 
and \eqref{eq:stocordermain1} is a generalization of Theorem 2.1 of Mulgund \cite{Mulgund26}.
\end{remark}

\subsection{Organization of the paper.}
In Section \ref{sec:globalstrongconvexity}, we present the global strong convexity properties of the power of the partition function for $q > 1$ {and $\beta\ne0$}. 
The proof of Theorem \ref{thm:covariance-convexity}
{uses} the variational approach{.}
In Section \ref{sec:stochordering}, we present the proofs of Theorem \ref{thm:generalversionMNT} and  Corollary \ref{cor:logconcave-product}. Finally, in Section \ref{sec:final}, we prove Theorem \ref{thm:simplex-phase}.

\section{Global strong covariance convexity.}
\label{sec:globalstrongconvexity}
In this section, for a fixed $q > 1$, we study the function $C \mapsto \E \left[ Z_\beta (X)^q \right]$, where $C \in \mathcal{E}_N$ and $X \sim \mathcal{N}(0, C)$, through a variational
approach and prove Theorem \ref{thm:covariance-convexity} and Theorem \ref{thm:simplex-stability-above-one}. The core of the argument is an evaluation of the second derivative of the functional along the classical interpolation path between {two centered Gaussian vectors}.
For $q > 1$ and $\beta \in {\R}$, set
\begin{align}\label{eq:mu-definition}
	\mu_{N,q,\beta} :=\begin{cases}
	\displaystyle	\frac{\beta^4q^2(q-1)^2}{4N}
		\exp \left(-\beta^2/2\right),
		&1<q<2, \\
		2^{q-3}\beta^4q(q-1),
		&q\ge2
	\end{cases}
\end{align}
and
\begin{align}\label{eq:eta-definition}
	\eta_{N,q,\beta} := \frac{\beta^2q(q-1)}2
	\E \left[
	 Z_\beta(S)^{q-2}
	\e^{\beta(S_1+ S_2)}
	\right]
\end{align}
where $S \sim \mathcal{N}(0, \Delta_N)$.
{In particular, $\mu_{N,q,0}=\eta_{N,q,0}=0$.}
\begin{theorem}
	\label{thm:second-derivative}
	Let $q>1$, $\beta \in \R$, and let $C_0,C_1\in\mathcal E_N$. For $t \in [0,1]$, put
	$$
	A = C_1-C_0,
	\quad
	C_t = (1-t)C_0+tC_1
	$$
	and denote by $X_t$ a centered Gaussian random vector with covariance matrix $C_t$.
	Then $t\mapsto\E [Z_\beta(X_t)^q]$ is twice differentiable on $[0,1]$ and
	\begin{align}\label{eq:global-strong-convexity-second}
		\frac{d^2}{dt^2}\E [Z_\beta(X_t)^q]
		\ge
		\mu_{N,q,\beta}\norm{A}_{{F}}^2. 
	\end{align}
\end{theorem}
The computation of various derivatives is based on  the {following} proposition.   
For a symmetric matrix $A$, {define the operator $\mathcal{L}_A$ on $C^2$ functions by} 
\begin{equation}
	\label{eq:lA}
	\mathcal{L}_A(g) := \frac{1}{2} \sum_{i,j = 1}^N A_{ij} \partial_{i,j}^2 g  = \frac{1}{2} \mathrm{tr}\left(A \, \mathrm{Hess}(g)\right).
\end{equation}
A function $f: \R^N \to \R$ is said to be  of moderate growth if for each $a > 0$, 
\[
\lim_{\|x\| \to + \infty} f(x) e^{-a \|x\|^2} = 0.
\]
\begin{proposition}
\label{prop:derivative}
Let $A$ be a symmetric matrix and $\Gamma_0$ be a {positive semidefinite} matrix. For every $t \in [0,1]$, set $\Gamma_t = \Gamma_0 + t A$ and assume that $\Gamma_t$ is {positive semidefinite}. 
Let $Z_t$ be a centered Gaussian vector with covariance matrix $\Gamma_t$.
Assume that $g: \R^N \to \R$ is a {$C^{2k}$ function}, $k \ge 1$, and that $g$ and all its {derivatives up to order} $2k$ are of moderate growth.
Set $F(t) := \E [g(Z_t)]$, for $t \in [0,1]$. {Then} $F$ is {$C^{k}$} and for every $r \in \{1, \ldots, k\}$,
\begin{equation}
	\label{eq:derivation}
	F^{(r)}(t) = 
	\E \left[ \underbrace{\mathcal{L}_A \circ \ldots \circ \mathcal{L}_A }_{r \ \mathit{times}}(g) \, (Z_t)\right].
\end{equation}
\end{proposition}
\begin{proof} 
{First assume that $\Gamma_t$ is positive definite for every $t\in[0,1]$.} By definition, one has for any $t \in [0,1]$,
\[
	F(t)  = \E [g(Z_t)] = \int_{\R^N} g(x) \rho_t(x) dx
\]
where
\[
\rho_t(x) = \frac{1}{(2\pi)^{N/2}} \frac{1}{\sqrt{\mathrm{det}(\Gamma_t)}}
\exp\left( - \frac{1}{2} \langle \Gamma_t^{-1} x , x \rangle \right),
\quad x \in \R^N.
\] 
We use the integral representation of $\rho_t$ (using the Fourier transform)
\[
	\rho_t(x) = \frac{1}{(2 \pi)^N} \int_{\R^N} e^{i \langle x, u \rangle} e^{-   \langle \Gamma_t u, u \rangle / 2 } \, du.
\]
to deduce that for every $t \in [0,1]$, and $x \in \R^N$,
\[
\frac{d}{dt} \rho_t(x) = \mathcal{L}_A(\rho_t) (x).
\]

Since $g$ has moderate growth, we can {differentiate $F$ under the integral sign} and this gives
\[
F'(t) = \int_{\R^N} g(x) \,  \frac{\partial}{\partial t} \rho_t(x) \, dx
= \int_{\R^N} g(x) \,  \mathcal{L}_A (\rho_t) (x) dx = 
\frac{1}{2} \sum_{i, j = 1}^N A_{ij} \int_{\R^N} g(x) \, \partial_{i,j}^2 (\rho_t)(x) \, dx.
\]
Since $g$ and its {derivatives} have moderate growth, {integrating by parts twice} proves that
\[
F'(t) = \frac{1}{2} \sum_{i, j = 1}^N A_{ij} \int_{\R^N} (\partial_{i,j}^2 g)(x) \, \rho_t(x) \, dx = \E \left[ \mathcal{L}_A(g)(Z_t)\right]
\]
which is the {claimed formula} for $r = 1$. A simple induction argument proves the result 
for every $r = 1, \ldots, k$.

{For the general case, let $G\sim\mathcal N(0,I_N)$ and, for $0\le r\le k$ and $0\le\varepsilon\le1$, set $g_r=\mathcal L_A^r (g)$ and}
\[
{H_r^\varepsilon(t)=\E\left[g_r\left((\Gamma_t+\varepsilon I_N)^{1/2}G\right)\right].}
\]
{The positive-definite case gives $(H_{r-1}^\varepsilon)'=H_r^\varepsilon$ for $\varepsilon>0$ and $1\le r\le k$. The covariance matrices are uniformly bounded, and moderate growth bounds each integrand, uniformly in $t$ and $\varepsilon$, by a constant multiple of $e^{\|G\|^2/4}$. Continuity of matrix square roots gives uniform convergence of the arguments for each fixed $G$. Dominated convergence therefore yields $H_r^\varepsilon\to H_r^0$ uniformly on $[0,1]$, and each $H_r^0$ is continuous. Passing to the limit in}
\[
{H_{r-1}^\varepsilon(t)-H_{r-1}^\varepsilon(0)=\int_0^t H_r^\varepsilon(u)\,du}
\]
{gives $(H_{r-1}^0)'=H_r^0$. Since $F=H_0^0$, this proves \eqref{eq:derivation} and $F\in C^k([0,1])$, with derivatives at the endpoints understood one-sidedly.}
\end{proof}
\begin{proof}[Proof of Theorem \ref{thm:second-derivative}.]
Let $C_0, C_1 \in \mathcal{E}_N$ be correlation matrices, and for $t \in [0,1]$, set
$
A := C_1 - C_0,$
$C_t := C_0 +tA,$
$X_t\sim {\mathcal N}(0,C_t)$ 
and
$$
F(t) := \E [Z_\beta(X_t)^q]
= \E \left[ \left( \sum_{i=1}^N e^{\beta (X_t)_i} \right)^q \, \right].
$$
Fix $x\in\R^N$, and write
$$
z_i = \e^{\beta x_i},\quad
S = \sum_{i = 1}^N z_i,\quad
p_i = \frac{z_i}{S}, \quad
g(x) = S^q.
$$
For any $q \ge 1$, the {hypotheses} of Proposition \ref{prop:derivative} are satisfied and it yields
\begin{align} \label{eq:soft-max-covariance-derivatives}
	F'(t) = \E[\mathcal L_A (g)(X_t)],\quad
	F''(t) = \E[\mathcal L_A^2 (g)(X_t)].
\end{align}

Set
$$
a = \langle Ap, p \rangle,\quad
h = \sum_i p_i(Ap)_i^2,\quad
k = \sum_{i,j}A_{ij}^2p_ip_j,\quad
\rho = q-2.
$$
Since ${C_0, C_1} \in \mathcal{E}_N$ are correlation matrices,
{we have} $A_{ii} = 0$ and the evaluation of $\mathcal{L}_A (g)$ requires only the computation of $\partial^2_{ij} g$ for $i \ne j$. Direct differentiation proves that
$$
\partial_i g = (\partial_i S^q) = q S^{q-1} \partial_i S = \beta q z_i S^{q-1}
\quad
\mathrm{and}
\quad
\partial_{ij}^2 g = (\partial_{ij}^2 S^q )= \beta^2 q(q-1) z_i z_j S^{q-2}, \ i \ne j.
$$
Therefore
\begin{align} \label{eq:soft-max-first-variation}
	\mathcal L_A (g) = \frac{\beta^2q(q-1)}2 \, S^q \, a.
\end{align}
Differentiating once more requires the computation of $\partial^2_{ij} \mathcal{L}_A (g)$ for $i \ne j$. Since
$$
\partial_{ij}^2 \left(S^q a\right) = \left(\partial_{ij}^2 S^q \right) a +  (\partial_i S^q )(\partial_j a )+ (\partial_j S^q) (\partial_i a) + S^q (\partial_{ij}^2 a),
$$
{we also need to evaluate the partial derivatives} of $a$. Noting that $a = \langle Az, z \rangle / S^2$ and that $(\partial_i S) = (\partial_i z_i) = \beta z_i$, one proves that
\[
(\partial_i a) = \beta \left( 2 \frac{z_i (Az)_i}{S^2} - 2 a  \frac{z_i}{S} \right)
\quad
\mathrm{and}
\quad
\left( \partial_j \frac{a}{S}\right) = \beta \left(  2 \frac{z_j (Az)_j}{S^3} -  3 a \frac{z_j}{S^2} \right), 
\]
and that
\[
(\partial_{ij}^2 a) = \beta^2 \left( 2 {A_{ij}} \frac{z_i z_j}{S^2}  - 4 \frac{z_i z_j ((Az)_i + (Az)_j)}{S^3} + 6 a \frac{z_i z_j}{S^2} \right),  \ i \ne j.
\]
Therefore, by definition of $\mathcal{L}_A$, one has
\begin{align*}
2 \mathcal L_A ( S^q \, a) & =  \sum_{i,j} A_{ij} \partial_{ij}^2 (S^q \, a) 
 = \sum_{i,j} A_{ij} \left( \left(\partial_{ij}^2 S^q \right) a +  (\partial_i S^q )(\partial_j a )+ (\partial_j S^q) (\partial_i a) + S^q (\partial_{ij}^2 a) \right)
\\
& = a  \sum_{i,j} A_{ij} \left(\partial_{ij}^2 S^q \right)   + 2 \sum_{i,j} A_{ij} (\partial_i S^q )(\partial_j a )
+ S^q  \sum_{i,j} A_{ij} (\partial_{ij}^2 a)
\\
& = \beta^2 S^q \left( q(q-1) a^2  + 2(2q h - 2 q a^2) +  ( 2 k - 8 h + 6 a^2) \right)
\\
& =  \beta^2 S^q \left( a^2(q-2)(q-3) +  4h(q-2) + 2k \right)
\end{align*}
which, {together with} \eqref{eq:soft-max-first-variation}, {gives}
\begin{align} \label{eq:soft-max-covariance-convexity}
	\mathcal L_A^2 (g) = \frac{\beta^4q(q-1)}4S^q\left( \rho(\rho-1)a^2+4\rho h+2k\right).
\end{align}

We record three elementary inequalities:
\begin{align} \label{eq:soft-max-ahk}
	k-2h+a^2\ge0,\quad
	a^2\le k,\quad
	a^2\le h.
\end{align}
Since the $p_i$'s are positive numbers that sum up to 1, {we can define $I,J$ to be independent, identically distributed random indices} with $\Prob(I = i) = p_i$, and put $d_i = (Ap)_i$. Then
$$
a = \E [A_{IJ}] = \E[d_I],\quad
k = \E [A_{IJ}^2],\quad
h = \E [d_I^2].
$$
The last two inequalities in \eqref{eq:soft-max-ahk} {follow} from the {Cauchy--Schwarz} inequality, while
$$
k-2h+a^2 = \E \left[ \left(A_{IJ}-d_I-d_J+a\right)^2\right]\ge0.
$$

If $1<q<2$, then $-1<\rho<0$. From \eqref{eq:soft-max-ahk},
$$
h\le\frac{k+a^2}{2}.
$$
Hence
\begin{align*}
	\rho(\rho-1)a^2+4\rho h+2k
	&\ge(\rho+1)(\rho a^2+2k) \ge(\rho+1)(\rho+2)k = q(q-1)k\ge0
\end{align*}
and \eqref{eq:soft-max-covariance-convexity} implies that 
\begin{align*}
	\E \left[ \mathcal{L}_A^2 (g)(X_t) \right] \ge \frac{\beta^4q^2(q-1)^2}4 \sum_{i \ne j} A_{ij}^2 
	\E \left[(p_i p_j S^q)(X_t) \right]. 
\end{align*}
To simplify the notation, {we suppress the dependence on $t$ from the notation}.
Observe that $p_i p_j S^q = z_i {z_j} S^{q-2}$ and
set $\theta := 2-q\in(0,1)$. By the
{Cauchy--Schwarz} inequality,
$$
\E[S^\theta] \, \E \left[\frac{z_i z_j}{S^\theta}\right]
\ge
\left(\E \left[ z_i^{1/2} z_j^{1/2} \right]\right)^2 .
$$
Since $\theta \in (0,1]$, one has $S^\theta \le \sum_{i=1}^N e^{\beta \theta X_i}$. The moment generating function of a centered Gaussian random variable $\eta$ is known to satisfy
$\E \left[ e^{\lambda \eta} \right]  = e^{\lambda^2 \E[\eta^2]/2}$, for all $\lambda \in \R$. Therefore, 
$$
\E \left[ z_i^{1/2} z_j^{1/2} \right] = e^{\beta^2 \E[(X_i + X_j)^2] / 8} \ge 1, 
\qquad i \ne j
$$
and
$$
\E[S^\theta] \le \sum_{i = 1}^N \E [e^{\beta \theta X_i}] = N e^{\beta^2 \theta^2 / 2} 
\le N e^{\beta^2 / 2}.
$$
Hence, for $1 < q < 2$, 
\begin{align}
	\label{eq:case1}
\E \left[\frac{z_i z_j}{S^\theta}\right] \ge e^{- \beta^2 / 2} / N
\qquad
\mathrm{and}
\qquad
	\E \left[ \mathcal{L}_A^2 (g)(X_t) \right] \ge \frac{\beta^4 q^2 (q-1)^2}{4N} e^{-\beta^2 /2} \, \|A\|_F^2.
\end{align}

If $q \ge 2$, then $\rho(\rho-1)a^2+4\rho h \ge 0${.} Indeed, if $q \ge 3$ then $\rho = q-2 \ge 1$ and all terms are nonnegative. {If} $2 \le q < 3$, $\rho(\rho - 1) \le 0$ and $a^2 \le h$ (see \eqref{eq:soft-max-ahk}) hence $\rho(\rho-1)a^2+4\rho h \ge \rho(\rho + 3) h \ge 0${.}
We deduce from \eqref{eq:soft-max-covariance-convexity} that 
\begin{align}
		\E \left[ \mathcal{L}_A^2 (g)(X_t) \right] \ge \frac{\beta^4q(q-1)}2 \E[S^q k] 
		 = \frac{\beta^4q(q-1)}2 \sum_{i \ne j}^N A_{ij}^2 \E[(p_i p_j S^q)(X_t)].
\end{align}
Since $S \ge z_i + z_j \ge 2 \sqrt{z_i \, z_j}$ and $q -2 \ge 0$, one has
\[
p_i p_j S^q = z_i z_j S^{q-2} \ge 2^{q-2} (z_i z_j)^{q/2}.
\]
Taking {expectations} and using again the fact that the moment generating
function of a centered Gaussian random variable is {at least 1}, we conclude that
\begin{align}
	\label{eq:case2}
	\E \left[ \mathcal{L}_A^2 (g)(X_t)\right] \ge 2^{q-3} \beta^4 q(q-1) \|A\|_F^2.
\end{align}
Combining \eqref{eq:soft-max-covariance-derivatives}, \eqref{eq:case1} and \eqref{eq:case2} 
shows \eqref{eq:global-strong-convexity-second}.
\end{proof}
\begin{proof}[Proof of Theorem \ref{thm:covariance-convexity}.]
With the {notation} of the theorem, set 
$F(t) := \E [Z_\beta(X_t)^q]$. Then by Theorem \ref{thm:second-derivative}, it is twice differentiable on $[0,1]$ and {satisfies} $F''(t) \ge \mu_{N,q,\beta} \|A\|_F^2 \ge 0,$ for every $t \in [0,1]$. Therefore, $F$ is convex and by {the second-order Taylor formula}
\[
  (1-t) F(0) + t F(1)  \ge F(t) + \frac{t(1-t)}{2} \mu_{N,q,\beta} \|A\|_F^2
\]
which {completes} the proof.
\end{proof}
\begin{proof}[Proof of Theorem \ref{thm:simplex-stability-above-one}.]
With the {notation} of the theorem, set $X \sim \mathcal{N}(0, C)$, $S \sim \mathcal{N}(0, \Delta_N)$, $C_0 := \Delta_N$, $C_1 := C$ and $A := C - \Delta_N$. Set 
$F(t)  := \E [Z_\beta(X_t)^q]$, where $X_t \sim \mathcal{N}(0, C_0 + t A)${. By} Theorem \ref{thm:second-derivative}, it is twice differentiable on $[0,1]$ and {satisfies} $F''(t) \ge \mu_{N,q,\beta} \|A\|_F^2 \ge 0,$ for every $t \in [0,1]$. Moreover{,} by \eqref{eq:soft-max-covariance-derivatives} and \eqref{eq:soft-max-first-variation}, one has
$$ 
F'(0) = \frac{\beta^2 q(q-1)}{2} 
\E \left[ Z_{\beta}(S)^{q-2} \sum_{i \ne j} A_{ij} e^{\beta (S_i + S_j)}\right].
$$
Since the law of $S$ is invariant under coordinate permutations, for every $i \ne j$,
\[
\E \left[ Z_{\beta}(S)^{q-2} e^{\beta (S_i + S_j)}\right]
=
\E \left[ Z_{\beta}(S)^{q-2} e^{\beta (S_1 + S_2)}\right]
\]
hence $F'(0) = \eta_{N,q,\beta} \sum_{i \ne j} A_{ij}$. Since $\Delta_N {\one} = 0$, one has
\[
F'(0) = \eta_{N,q,\beta} \langle A {\one}, {\one} \rangle
= 
\eta_{N,q,\beta} \langle C {\one}, {\one} \rangle.
\]
By {the second-order Taylor formula}, one has
\[
F(1) - F(0) = F'(0) + \int_0^1 (1-t) F''(t) dt 
\ge 
\eta_{N,q,\beta} \langle C {\one}, {\one} \rangle
+ \frac{\mu_{N,q,\beta}}{2} \|C - \Delta_N \|_F^2.
\]
{This completes the proof.}
\end{proof}

\section{Stochastic ordering for log-concave {profiles}.}
\label{sec:stochordering}
The goal of this section is to prove Theorem \ref{thm:generalversionMNT} and Corollary \ref{cor:logconcave-product}.  We 
recall basic facts {about Legendre duality and inf-convolution}. 
Let $w : \R \to {(0,+\infty)}$ be a {log-concave} function  and set
\begin{equation}
\label{eq:defphipsi}
\varphi(u) := \log \left( \int_\R w(u-x) d\gamma(x) \right),
\qquad
\psi(s) := \inf_{t \in \R} \left\{ \frac{t^2}{2} + \varphi(s-t) \right\}.
\end{equation}
Since $w$ is log-concave, there {exist} $\alpha, \beta \in \R$ such that $w(z) \le e^{\alpha z + \beta}$ and by properties of the convolution, $\varphi$ is a concave $C^2$ function on $\R$. {For each fixed $t$, the function $s\mapsto t^2/2+\varphi(s-t)$ is concave. Therefore the inf-convolution $\psi$, being the pointwise infimum of these concave functions, is concave on its effective domain $\{s:\psi(s)>-\infty\}$.}
\begin{proposition}
\label{prop:wellknown}
Assume {that} $w : \R \to {(0,+\infty)}${. Then}
the functions $\varphi$ and $\psi$ satisfy the following properties.
\begin{enumerate}[(i)]
	\item For any $u \in \R$, $\varphi(u) = \sup_{s \in \R} \{ \psi(u-s) - {s^2 /2}\}$.
	\item For any $s \in \R$,
	\begin{align}
		\label{eq:defpsi}
		\psi(s)  =  \inf_{t \in \R} \log\left( \int_{\R} w(s-y) e^{ty} d\gamma(y)\right)
		=  \log\left( \int_{\R} w(s-y) e^{t^*(s) \, y} d\gamma(y)\right)
	\end{align}
\end{enumerate}
where $t^*(s)$ is a unique real number {satisfying}
\begin{equation}	
	\label{eq:propt*}
	\int_{\R} y \, w(s-y) \, e^{t^*(s) \, y} d\gamma(y) =  0
\end{equation}
\begin{equation}
	\label{eq:propt*2}
	t^*(s) = \varphi'(s - t^*(s)) = \psi'(s).
\end{equation}
\end{proposition}

\begin{proof}
Define a function $H$ by $H(u) = u^2/2 + \varphi(u)$ for any $u \in \R$, and observe that ({using the change of variables $z=u-x$})
\[
H(u) = \log \left(e^{u^2/2} \int_{\R} w(u-x) \,  d\gamma(x) \right)= \log \left( \int_{\R} w(z) \, e^{uz} d\gamma(z) \right)
\]
meaning that it is {the log-Laplace transform, up to normalization, of the measure with density proportional to}
$w(z) e^{-z^2/2}$. Hence $H$ is a convex function. Its Legendre transform satisfies
\begin{align*} 
	H^{*}(s)  = \sup_{u\in\R} \left\{us-\frac{u^2}{2}-\varphi(u)\right\} 
	 = \frac{s^2}{2} - \inf_{u\in\R} \left\{\frac{(u-s)^2}{2}+\varphi(u)\right\} 
	 = \frac{s^2}{2}-\psi(s).
\end{align*}
Since $H$ is convex on its support, the bipolar {theorem} asserts that
$H = H^{**}$. Hence
\begin{align*}
	\varphi(u) & = \sup_{r\in\R} \left\{ur-\frac{r^2}{2}+\psi(r)\right\} -\frac{u^2}{2} = \sup_{r\in\R} \left\{ \psi(r)-\frac{(u-r)^2}{2} \right\}
\end{align*}
which is exactly property $(i)$. 
For fixed $s\in\R$, set
$
\Theta_s(t) := \frac{t^2}{2}+\varphi(s-t).
$
Using
$\e^{ty}d\gamma(y) = \e^{t^2/2}d\gamma(y-t),$
we obtain 
\begin{align*} 
	\Theta_s(t) = \log\int_{\R}w(s-y)\e^{ty}d\gamma(y)
\end{align*}
The function $\Theta_s$ is also the log of the Laplace transform of a random variable and it is not difficult to see that  it is strictly convex in $t$. Since
$w >0$ on $\R$, we know by log-concavity of $w$ that on any {interval}
$[a,b] \subset \R$, we have
\[
\min_{z \in [a,b]} w(z) = \min\{w(a), w(b)\} > 0.
\] 
Hence the limit of $\Theta_s(t)$ equals $+ \infty$ as $|t|$ tends to infinity and 
$\Theta_s(t)$ {attains its minimum on $\R$ at a unique point} $t^*(s)$.  Equation \eqref{eq:defpsi} follows immediately and
$t^*(s)$ is 
characterized by the equation $\Theta_s'(t^*(s)) = 0$. Since 
$$
\Theta_s' (t) = 
\frac{ \int_{\R} y\, w(s-y)\e^{ty}d\gamma(y)}{\int_{\R}w(s-y)\e^{ty}d\gamma(y)},
$$ 
equation \eqref{eq:propt*} is satisfied.
By definition of $\Theta_s$, we have $t^*(s) = \varphi'(s- t^*(s))$ which proves the first part of equation \eqref{eq:propt*2}. By the implicit function theorem, we know that 
$t^*$ is a $C^1$ function.  The second part {follows by differentiating the relation} $\psi(s) = (t^*(s))^2/2 + \varphi(s-t^*(s))$.
\end{proof}

{We now recall} the centered reverse
{Brascamp--Lieb} inequality of {Nakamura--Tsuji}~\cite{NT26} ({which} was further extended in \cite{MNT26}). As explained by Mulgund \cite{Mulgund26}, Lemma  {\ref{lem:centered-product}} is a particular case of Theorem 1.10 in \cite{MNT26}. {It is also an application} of Theorem 2.4 in \cite{NT26}, with the same choice of {Brascamp--Lieb} datum. 
\begin{lemma}[Centered Gaussian product inequality \cite{NT26}] 
	\label{lem:centered-product}
	Let $R$ be a correlation matrix, let $X\sim {\mathcal N}(0,R)$, and let
	$f_i:\R\to[0,\infty)$ be  log-concave functions such that
	$$
	\int_\R x f_i(x) d\gamma(x) = 0,
	\quad 1\le i\le N.
	$$
	Then
	\begin{align}\label{eq:centered-product} 
		\E \left[ \prod_{i = 1}^N f_i(X_i) \right]
		\ge \prod_{i = 1}^N \int_\R f_i d\gamma.
	\end{align}
\end{lemma}
The next statement is a consequence of Lemma {\ref{lem:centered-product}} and the 
{facts recalled above about inf-convolution}.
We consider {log-concave functions $w_i:\R\to(0,+\infty)$, $1\le i\le N$}. To each $w_i$, we associate  the functions $\varphi_i, \psi_i$  defined as in \eqref{eq:defphipsi}.
To every $s=(s_1, \ldots, s_N) \in \R^N$, we associate {$t^*(s) = (t_1^*(s_1), \ldots, t_N^*(s_N))\in \R^N$} where ${t_i^*(s_i)}$ is the unique real number defined in Proposition \ref{prop:wellknown} with the functions $\varphi_i$ (and $\psi_i$). 
\begin{proposition}
\label{prop:step2}
Let $X \sim \mathcal{N}(0, R)$ be a centered Gaussian vector where $R \in \mathcal{E}_N$. Then
\[
\E \left[ \prod_{i=1}^{N} w_i\left(s_i - \left(R t^*(s)\right)_i - X_i\right) \right] 
\ge
\exp\left(- \frac{1}{2} \langle R t^*(s), t^*(s) \rangle + \sum_{i=1}^N \psi_i(s_i)\right).
\]
\end{proposition}
\begin{proof}
{For $y\in\R$, set} $f_i(y) = w_i(s_i - y) \, e^{{t_i^*(s_i)} \, y}$. Since $w_i$ is log-concave, $f_i$ is log-concave. From \eqref{eq:propt*}, $f_i$ has Gaussian barycenter at the origin{. Hence, by} Lemma {\ref{lem:centered-product}}
\begin{equation}
	\label{eq:keyineq}
	\E \left[ \prod_{i=1}^{N} f_i(X_i) \right] 
	\ge 
	\prod_{i=1}^{N} \int_{\R} f_i(y) d\gamma(y) = \exp\left( \sum_{i=1}^{N} \psi_i(s_i)\right),
\end{equation}
where the last equality follows from \eqref{eq:defpsi}. Moreover, by definition of the $f_i$'s,
\[
\prod_{i=1}^{N} f_i(X_i) = \exp \left( \langle t^*(s), X \rangle \right) \prod_{i=1}^{N} w_i(s_i - X_i).
\]
However, for every {nonnegative} function $F$, one has the Gaussian shift {identity}
\[
\E \left[ \exp \left( \langle t^*(s), X \rangle \right) F(X) \right]
=
\exp\left( \frac{1}{2} \langle R t^*(s), t^*(s) \rangle \right) \E \left[  F(X + Rt^*(s)) \right]{.}
\]
{Therefore,}
\[
\E \left[ \prod_{i=1}^{N} f_i(X_i)\right]  = \exp\left( \frac{1}{2} \langle R t^*(s), t^*(s) \rangle \right)
\E \left[ \prod_{i=1}^{N} w_i\left(s_i - \left(R t^*(s)\right)_i - X_i\right) \right],
\]
which, combined with \eqref{eq:keyineq},  {completes} the proof.
\end{proof}
\begin{proof}[Proof of Theorem \ref{thm:generalversionMNT}.]
The proof is based on a {suitable} optimization problem.
For any $y \in \R^N$, set
\begin{equation}
\label{eq:deffunc}
\mathcal{J}(y) := - \frac{\|y\|^2}{2} + \sum_{i=1}^{N} \psi_i \left( (R^{1/2} \, y)_i\right).
\end{equation}
Since {the functions $\psi_i$} are concave, $\mathcal{J}$ is strictly concave and {$\mathcal{J}(y)\to-\infty$ as $\|y\|\to\infty$}. Therefore the functional $\mathcal{J}$ attains its maximum at a unique point that we call $y^*$. It satisfies $\nabla \mathcal{J}(y^*) = 0$. Since $R$ is symmetric, the computation of the gradient yields the  equation
\[
y^* = R^{1/2} Y, \quad \mathrm{where} \quad Y_i = \psi'_i\left( (R^{1/2} y^*)_i \right).
\]
We define $s_i$ by $s_i = (R^{1/2} y^*)_i$. By Proposition \ref{prop:wellknown}, we get that the associated ${t_i^*(s_i)}$ satisfies $\psi_i'(s_i) = {t_i^*(s_i)}$ (see {\eqref{eq:propt*2}}). Therefore
\[
y^* = R^{1/2} t^*(s),
\qquad
\qquad
s_i = (R^{1/2} y^*)_i = \left( R t^*(s) \right)_i
\qquad
\mathrm{and}
\qquad
\langle R t^*(s), t^*(s) \rangle = \|y^*\|^2.
\]
By Proposition \ref{prop:step2}, we conclude that
\[
\E \left[ \prod_{i=1}^{N} w_i(-X_i) \right] \ge \exp\left( - \frac{\|y^*\|^2}{2} + \sum_{i=1}^N \psi_i \left( (R^{1/2} \, y^*)_i\right)\right) = {\exp\left(\sup_{y \in \R^N} \mathcal{J}(y)\right)}
\]
which is exactly the {claimed inequality}  \eqref{eq:variational} in Theorem \ref{thm:generalversionMNT} (after observing that the Gaussian vector $X$ is symmetric).
\end{proof}
Finally, our approach {shows} that the assumption $R - \frac{1}{N} J_N \succeq 0$ is used only at the very end.
This assumption  {could} therefore be relaxed{,} and {other} corollaries of Theorem \ref{thm:generalversionMNT} {could be derived}. 
For the purpose of this paper, {we restrict attention to} the consequences in the 
case $R - \frac{1}{N} J_N \succeq 0$ {as in} Corollary \ref{cor:logconcave-product}.
\begin{lemma}
	\label{lem:linearalgebra}
A correlation matrix $R$ is such that $R - \frac1N J_N \succeq 0$ if and only if there exists a vector $v \in \R^N$ such that
\[
R^{1/2} \, v = {\one}
\qquad
\mathit{and}
\qquad
\|v\| \le \sqrt N.
\]
\end{lemma}
\begin{proof}
Let $T = R^{1/2}${. Then} the matrix inequality is equivalent to
$$
|\ip{x}{\one}|\le\sqrt N\norm{Tx},\quad x\in\R^N.
$$
Hence $L(Tx) := \ip{x}{\one}$ is a well-defined linear functional on
$\operatorname{Ran}T$ with $\norm{L}\le\sqrt N$. By the Riesz representation
theorem, there is $v\in\operatorname{Ran}T$ with $\norm{v}\le\sqrt N$ such that
$$
\ip{Tx}{v} = \ip{x}{\one},\quad x\in\R^N.
$$
Since $T$ is symmetric, $Tv = \one$. Conversely, if $Tv = \one$ and
$\norm{v}\le\sqrt N$, then
$$
|\ip{x}{\one}|^2 = |\ip{Tx}{v}|^2\le N\norm{Tx}^2 = N\ip{Rx}{x},
$$
which is equivalent to $R-\frac{1}{N} J_N\succeq0$.
\end{proof}

\begin{proof}[Proof of Corollary \ref{cor:logconcave-product}.]
Let $v$ be defined in Lemma \ref{lem:linearalgebra}{.} {Choose} $y = tv$ where $t \in \R${.}
{By} Theorem \ref{thm:generalversionMNT} {and} \eqref{eq:variational}, {we obtain}
\[
\E \left[  \prod_{i=1}^{N} w_i(X_i) \right] \ge {\exp\left(\sup_{t \in \R} \left( - \frac{N t^2}{2} + \sum_{i=1}^N \psi_i \left( t\right)\right)\right)}
\]
since for every $i$, $(R^{1/2} \, v)_i = 1$ and $\|v\|^2 \le N$. This is the {claimed inequality} in 
\eqref{eq:stocorder2}.

Assume first that {$w_1=\cdots=w_N=w$, where $w : \R \to (0,+\infty)$ is log-concave}.  Then {$\psi_1=\cdots=\psi_N=\psi$} (associated {with} $w$) and \eqref{eq:stocorder2} implies that 
\[
\E \left[ \prod_{i=1}^{N} w(X_i) \right] \ge {\exp\left(\sup_{t \in \R} \left( - \frac{N t^2}{2} +  N  \psi \left( t\right)\right)\right)}
= \exp(N \varphi(0))
\]
by assertion $(i)$ in Proposition \ref{prop:wellknown}. The definition of $\varphi$ and {Gaussian symmetry imply} that
\[
\varphi(0) = {\log}\E [w(G_1) ],
\]
which {completes} the proof of \eqref{eq:stocordermain1} when $w >0$ on $\R$. {Now let $w:\R\to[0,+\infty)$ be log-concave. If $w=0$ almost everywhere, both sides of \eqref{eq:stocordermain1} vanish because each marginal is a nondegenerate Gaussian. Otherwise, $w$ is positive on an interval of nonzero length. For $\varepsilon>0$, we set}
$w_\varepsilon = w \star \gamma_\varepsilon$ where $\gamma_\varepsilon (t) = \frac{1}{\varepsilon \sqrt{2\pi}} \exp \left( -\frac{t^2}{2 \varepsilon^2}  \right) $.
Since log-concave functions are stable under convolution, $w_\varepsilon$ is log-concave on $\R$ and $w_\varepsilon : \R \to {(0,+\infty)}$. 
{Choose $a,b\in\R$ such that $w(x)\le e^{ax+b}$. For $0<\varepsilon\le1$,}
\[
{w_\varepsilon(x)\le e^{ax+b+a^2\varepsilon^2/2}\le e^{ax+b+a^2/2}.}
\]
{These bounds, and their products, are integrable against the Gaussian laws in question. The approximate-identity property gives $w_\varepsilon(x)\to w(x)$ for Lebesgue-almost every $x$. Since all marginals have densities, the convergence holds almost surely at each $X_i$ and at $G_1$, even if $R$ is singular. Dominated convergence therefore gives, as $\varepsilon\to0$,}
\[
\E \left[ \prod_{i=1}^{N} w_\varepsilon (X_i) \right] \longrightarrow \E \left[ \prod_{i=1}^{N} w (X_i) \right]
\quad
\mathrm{and}
\quad
\E [w_\varepsilon (G_1)] \longrightarrow \E [w(G_1)]
\]
which proves assertion \eqref{eq:stocordermain1} in the general case $w : \R \to {[0,+\infty)}$.

For every $\lambda > 0$, 
inequality \eqref{eq:LaplaceIneq} is a direct consequence of \eqref{eq:stocordermain1} by taking  the Gumbel profile 
 given by $w (x) = \exp\left( -\lambda e^{\beta x} \right)$.
\end{proof}

\section{Concluding the proof of Theorem \ref{thm:simplex-phase}.}
\label{sec:final}
We conclude the paper by proving  Theorem \ref{thm:simplex-phase}. Recall that $X \sim \mathcal{N}(0, C)$ and $S \sim \mathcal{N}(0, \Delta_N)$ are two centered Gaussian {vectors. Here} $C$ is a correlation matrix{,} and $\Delta_N$ is defined by \eqref{eq:simplex-covariance}.
{If $\beta=0$, then $Z_0\equiv N$, so all asserted inequalities are equalities. Henceforth assume $\beta\ne0$.}
At $q=1$, every marginal has variance 1, so
\[
\E [ \, |Z_\beta(X) | \, ]= \sum_{i = 1}^N \E \left[ e^{\beta X_i} \right]
= N e^{\beta^2 /2}.
\]

There are two different mechanisms depending on whether $q>1$ or $q<1$. 
In the case $q> 1$, Theorem \ref{thm:simplex-stability-above-one} gives
\[
\E [Z_\beta(X)^q] \ge \E [Z_\beta(S)^q]
\]
with equality if and only if the covariance matrix of $X$ is equal to $\Delta_N$.  {This proves the assertion for $q>1$.}
In the case $q<1$, we decompose the proof into two steps. 
The {first} is a transformation that was already noticed in \cite{KLZ2017, GS26, Mulgund26}.
Let $\eta  \sim \mathcal{N}(0, \frac{1}{N-1})$ be a Gaussian random variable, independent of {both} $X$ and $S${,} and define the Gaussian vector $Y$  by 
\[
Y_i = \sqrt{\frac{N-1}{N}} (X_i + \eta), \quad \mathrm{and} \quad G_i = \sqrt{\frac{N-1}{N}} (S_i + \eta).
\]
Then $Y \sim \mathcal{N}(0, R)$ where $R$ is a correlation matrix such that 
\[
R - \frac{1}{N} J_N = \frac{N-1}{N} \, C \succeq 0,
\]
while $G \sim \mathcal{N}(0, I_N)$, {by} the definition of $\Delta_N$.
Therefore, the assumption of Corollary \ref{cor:logconcave-product} {is satisfied} and  \eqref{eq:LaplaceIneq} {holds}.

A simple algebraic computation shows that
for every $\beta \in \R$,
\[
Z_{\tilde{\beta} }(Y) = e^{\beta \eta} Z_\beta (X)
\quad
\mathrm{and}
\quad
Z_{\tilde{\beta} }(G) = e^{\beta \eta} Z_\beta (S)
\]
where {$\tilde{\beta}=\beta\sqrt{N/(N-1)}$}. {Since $\eta$ is independent of both $X$ and $S$}, taking {expectations} shows that 
for every $q \ne 0$,
\begin{equation}
	\label{eq:q-power}
\E [Z_{\tilde{\beta} }(Y)^q ] = \E [e^{q \beta \eta}] \, \E [Z_\beta (X)^q],
\quad
\quad
\E [Z_{\tilde{\beta} }(G)^q] = \E [e^{q \beta \eta}] \, \E [Z_\beta (S)^q]
\end{equation}
and
\begin{equation}
	\label{eq:q=0}
\E [ \log Z_{\tilde{\beta} }(Y) ] =  \E [ \log Z_\beta (X) ],
\quad
\quad
\E [ \log Z_{\tilde{\beta} }(G) ] =  \E [ \log Z_\beta (S) ].
\end{equation}
The idea guiding the second step in the case $q<0$ is the {well-known connection} between completely monotone functions ({see, e.g., Bernstein's theorem} \cite{Bernstein1928})
and Laplace
transform ordering. 
In the case $q < 0$, we use the fact that the power function $x \mapsto x^q$ is a completely monotone function on $(0, + \infty)$. More explicitly, 
\[
x^q = \frac{1}{\Gamma(-q)} \int_{0}^{+\infty} s^{-q -1} e^{-sx} ds
\]
{By Tonelli's theorem and} \eqref{eq:LaplaceIneq}, we get that
$\E [Z_{\tilde{\beta} } (Y)^q] \ge \E [Z_{\tilde{\beta} }(G)^q]$. By \eqref{eq:q-power}, one has 
$\E [Z_\beta (X)^q] \ge \E [Z_\beta (S)^q]$.
Since $q<0$, raising the inequality to the power $1/q$ 
proves the result.

If $q=0$, we use the following {well-known identity for the logarithm}
\[
\log(x) = \int_0^{+\infty} \frac{e^{-s} - e^{-sx}}{s} ds{.}
\]
{For any Gaussian vector $V\in\R^N$ and $b\in\R$,}
\[
{|\log Z_b(V)|\le\log N+|b|\max_{1\le i\le N}|V_i|,}
\]
{so $\E|\log Z_b(V)|<\infty$. Moreover, for $x>0$,}
\[
{\int_0^{+\infty}\frac{|e^{-s}-e^{-sx}|}{s}\,ds=|\log x|.}
\]
{Thus Fubini's theorem applies to the logarithmic representation. Together with \eqref{eq:LaplaceIneq}, this yields}
\[
\E [\log Z_{\tilde{\beta} } (Y)] \le \E [\log Z_{\tilde{\beta} } (G)].
\]
This proves the result, using \eqref{eq:q=0}.

If $0 < q < 1$, we use the identity
\[
x^q = \frac{q}{\Gamma(1-q)} \int_0^{+\infty} (1 - e^{-sx}) \frac{ds}{s^{q+1}}
\]
{and Tonelli's theorem together with} \eqref{eq:LaplaceIneq} to get that
$\E [Z_{\tilde{\beta} } (Y)^q] \le \E [Z_{\tilde{\beta} }(G)^q]$. We conclude using \eqref{eq:q-power}.

\bibliographystyle{abbrv} \bibliography{biblio}

\address
\end{document}